\documentclass[openany]{amsart}
\usepackage{amssymb, amsfonts, lacromay}
\usepackage{mathrsfs,comment}
\usepackage[usenames,dvipsnames]{color}
\usepackage[normalem]{ulem}
\usepackage{url}
\usepackage{tikz-cd}
\usepackage[all,arc,2cell]{xy}
\UseAllTwocells
\usepackage{enumerate}             

\usepackage{todonotes}
\usepackage{hyperref}  
\hypersetup{%
  bookmarksnumbered=true,%
  bookmarks=true,%
  colorlinks=black,%
  linkcolor=true,%
  citecolor=true,%
  filecolor=true,%Ä
  menucolor=black,%
  pagecolor=black,%
  urlcolor=black,%
  pdfnewwindow=true,%
  pdfstartview=FitBH}

\def\makeautorefname#1#2{\expandafter\def\csname#1autorefname\endcsname{#2}}
\def\equationautorefname~#1\null{(#1)\null}
\makeautorefname{footnote}{footnote}%
\makeautorefname{item}{item}%
\makeautorefname{figure}{Figure}%
\makeautorefname{table}{Table}%
\makeautorefname{part}{Part}%
\makeautorefname{appendix}{Appendix}%
\makeautorefname{chapter}{Chapter}%
\makeautorefname{section}{Section}%
\makeautorefname{subsection}{Section}%
\makeautorefname{subsubsection}{Section}%
\makeautorefname{theorem}{Theorem}%
\makeautorefname{thm}{Theorem}%
\makeautorefname{sta}{Statement}%
\makeautorefname{cor}{Corollary}%
\makeautorefname{lem}{Lemma}%
\makeautorefname{prop}{Proposition}%
\makeautorefname{pro}{Property}
\makeautorefname{conj}{Conjecture}%
\makeautorefname{defn}{Definition}%
\makeautorefname{notn}{Notation}
\makeautorefname{notns}{Notations}
\makeautorefname{rem}{Remark}%
\makeautorefname{rems}{Remarks}%
\makeautorefname{quest}{Question}%
\makeautorefname{exmp}{Example}%
\makeautorefname{ax}{Axiom}%
\makeautorefname{claim}{Claim}%
\makeautorefname{assn}{Assumption}%
\makeautorefname{asses}{Assumptions}%
\makeautorefname{countcom}{Assumption}%
\makeautorefname{countcoms}{Assumptions}%
\makeautorefname{countmcom}{Assumption}%
\makeautorefname{con}{Construction}%
\makeautorefname{prob}{Problem}%
\makeautorefname{warn}{Warning}%
\makeautorefname{obs}{Observation}%
\makeautorefname{conv}{Convention}%

\newtheorem{thm}{Theorem}[section]

\newtheorem{cor}{Corollary}[section]
\newtheorem{prop}{Proposition}[section]
\newtheorem{lem}{Lemma}[section]

\theoremstyle{definition}
\newtheorem{defn}{Definition}[section]

\newtheorem{con}{Construction}[section]
\newtheorem{exmp}{Example}[section]
\newtheorem{notn}{Notation}[section]

\newtheorem{rem}{Remark}[section]

\newcounter{assn}
\renewcommand{\theassn}{\Alph{assn}}
\newenvironment{ass}[1][]{\refstepcounter{assn}\par\medskip\noindent
   \textbf{Assumption~\theassn.} \rmfamily}{\medskip}

 \newcounter{countcom}
\renewcommand{\thecountcom}{\Alph{countcom}$_{com}$}

    \newcounter{countmcom}
\renewcommand{\thecountmcom}{\Alph{countmcom}$_{mcom}$}

\makeatletter
\let\c@obs=\c@thm
\let\c@sta=\c@thm
\let\c@cor=\c@thm
\let\c@prop=\c@thm
\let\c@lem=\c@thm
\let\c@prob=\c@thm
\let\c@con=\c@thm
\let\c@conj=\c@thm
\let\c@defn=\c@thm
\let\c@notn=\c@thm
\let\c@notns=\c@thm
\let\c@exmp=\c@thm
\let\c@ax=\c@thm
\let\c@pro=\c@thm
\let\c@ass=\c@thm
\let\c@warn=\c@thm
\let\c@rem=\c@thm
\let\c@asscom=\c@thm
\let\c@assmcom=\c@thm
\let\c@conv=\c@thm
\let\c@sch=\c@thm
\let\c@equation\c@thm
\numberwithin{equation}{section}
\makeatother

\definecolor{orange}{rgb}{1,0.5,0}

\newcommand{\bq}{\mathbf{q}}

\title{The homotopical monadicity theorem}

\author{Hana Jia Kong}
\address{School of Mathematical Sciences, Zhejiang University, Hangzhou, China}

\email{hana.jia.kong@gmail.com}

\author{J. Peter May}
\address{Department of Mathematics, The University of Chicago, Chicago, IL 60637}
\email{may@math.uchicago.edu}

\author{Foling Zou}
\address{Institute of Mathematics, Chinese Academy of Sciences, Beijing, China}
\email{zoufoling@amss.ac.cn}


\begin{document}

\begin{abstract}  We give a broad homotopical analog of the classical categorical Beck monadicity theorem.  It often holds when classical monadicity fails.  This grew out of an understanding of a general context for recognition principles in iterated loop space theory, as treated in \cite{KMZ1},  but the present result applies differently and more generally.   An example gives a new perspective on the old equivalence between simplicial sets and topological spaces: both are equivalent to simplicial topological spaces, and the equivalence implies a curiously close relationship between realizations of simplicial spaces and realizations of their underlying simplicial sets, viewed as discrete simplicial spaces.
\end{abstract}

\maketitle

\tableofcontents

%\mainmatter

\section*{Introduction}  Throughout this paper, we assume given an adjunction $\SI\dashv \OM$
between categories $\sT$ and $\sS$.   The composite $\GA = \OM \SI$ is  then a monad in $\sT$ that acts naturally on $\OM Z$ for objects $Z$ of $\sS$, so that $\OM$ factors through a functor $\OM_{\GA}\colon \sS \rtarr \GA[\sT]$, where $\GA[\sT]$ is the category of $\GA$-algebras in  $\sT$.   As we explain in \autoref{SecClass}, $\OM_{\GA}$ generally has a left adjoint $\SI_{\GA}$, and Beck's  categorical monadicity theorem gives conditions which ensure that  
$(\SI_{\GA},\OM_{\GA})$ is an adjoint equivalence.  While that theorem has many applications, we are concerned with what happens when it fails, as is more usual.

 In \autoref{SecHom}, we give minimal  homotopical conditions on $\sT$ and $\sS$ which ensure that $\GA[\sT]$ and $\sS_{c}$ have equivalent homotopy categories, where $\sS_{c}$ is the full subcategory of ``$\OM$-connective" objects of $\sS$ (as defined in \autoref{Zbar}).    In some applications, $\sS_{c} = \sS$ (see \autoref{all?}).  Our first version of such an equivalence is the 
``crude homotopical monadicity theorem"  (see \autoref{HomMon}).  The equivalence relies on the use of a well-behaved natural homotopical resolution $\overline{Y} = B(\GA, \GA, Y)$ of a $\GA$-algebra $Y$. However, we do not know that $\overline{Y}$ is itself a $\GA$-algebra, despite its role in the proof.  The reason is that $\GA$ does not commute with realization.  

As we observe in \autoref{OPLAX}, the two weak equivalences with source $\overline{Y}$ that appear in our comparison of a $\GA$-algebra $Y$ to a $\GA$-algebra  $\OM Z$  nevertheless exhibit preservation of monadic algebra structure.  This is  due to a conceptual link between monads appearing in the input and output of realization.  Weakening the notion of a map of $\GA$-algebras accordingly, our axioms imply the desired equivalence of homotopy categories, where the homotopy category of $\GA$-algebras is defined using the weakened notion of a map of $\GA$-algebras (see \autoref{HomMon2}, \autoref{Onew}, and \autoref{Onew2}). An abstract categorical framework that we are applying is presented in the brief Appendix.  We will discuss passage to $\infty$-categories in \cite{KMZ1}. 

In \autoref{SecSimp} we give a perhaps whimsical application where there presumably is no such monad $\bC$.   It shows that the homotopy category of simplicial spaces sits in between the homotopy categories of simplicial sets and spaces and is equivalent to both and also to the homotopy categories of algebras over both of the two relevant adjunction monads.  The first should be an old result but in our formulation seems to be new.  Consideration of these particular adjunction monads appears to be new, and we admit to having no applications in mind. The section raises unanswered questions about the relationship between realization of simplicial spaces and realization of their underlying simplicial sets. 

Since the axioms are so simple and so easily verified, we expect there to be many  examples of homotopical monadicity in fields other than algebraic topology where homotopy categories play a role.   The axiomatization in this paper has already been applied to motivic homotopy theory, by Srinivasan \cite{Ajay}.  

The results here have earlier $\infty$-categorical counterparts, known as Barr-Beck-Lurie monadicity \cite{Lurie}, but the axioms there are verified quite differently.  
  
\section{The categorical monadicity theorem}\label{SecClass}
\subsection{The monadic adjunction associated to an adjunction} Recall that for any adjunction $(\SI,\OM)$ from any category $\sT$ to another category $\sS$, with unit $\et$ and counit $\epz$, the isomorphism 
\begin{equation}\label{adj1} 
\sS(\SI X, Y) \iso \sT(X, \OM Y)
\end{equation} 
is given by the composite
\begin{equation}\label{adj2} 
\xymatrix@1{  \sS(\SI X, Y)  \ar[r]^-{\OM} & \sT(\OM\SI X, \OM Y) \ar[r]^-{\et^*} & \sT(X, \OM Y)\\}
\end{equation}
with inverse the composite
\begin{equation}\label{adj3} 
\xymatrix@1{ \sT(X, \OM Y) \ar[r]^-{\SI} & \sS(\SI X,\SI\OM Y) \ar[r]^-{\epz_*} & \sS(\SI X,Y).\\}
\end{equation} 
The composite $\GA =\OM\SI$ is a monad with  product $\mu = \OM\epz\colon \GA\GA \rtarr \GA$ and unit $\et\colon \Id \rtarr \GA$.  For $Z\in \sS$, $\OM Z$ is a  $\GA$-algebra with action  $\OM\epz\colon \GA \OM Z \rtarr \OM Z$.   

\begin{notn}   Let $\GA[\sT]$ denote the category of $\GA$-algebras in $\sT$ and $\OM_{\GA}$ denote the functor $\OM$, but viewed as taking values in $\GA[\sT]$.   We write $\GA$ ambiguously both for the monad in $\sT$ and for the  functor $\OM_{\GA}\SI$ to $\GA$-algebras.
\end{notn}

Here we depart from the categorical literature.   As we will want anyway when we turn to homotopy theory, we assume that $\sS$ and $\sT$ are cocomplete.   However, all that  we really need for the moment is the existence of certain coequalizers.  Recall that a coequalizer of maps  $f,g\colon  A\rtarr B$  in $\sT$ is a map $q\colon B\rtarr C$ in $\sT$ such that $qf = qg$ and $q$ is universal with this property.  

We shall generalize the context above and the  following definitions and results in \cite[Section 1.1]{KMZ1}.   Proofs omitted here are given there. 

\begin{defn}\label{twist} Define $\SI_{\GA}\colon \GA[\sT] \rtarr \sS$ to be the tensor product $\SI\otimes_{\GA} (-)$.  Explicitly, on a $\GA$-algebra $Y$, it is given by  the coequalizer in $\sS$ displayed in the diagram
\begin{equation}\label{split4} 
 \xymatrix@1{
\SI {\GA} Y  \ar@<.7ex>[rr]^{\epz}  \ar@<-.7ex>[rr]_{\SI\tha}  &&  \SI  Y \ar[rr]^-{q} & & \SI_{\GA}Y. \\}  
\end{equation}
\end{defn}

By comparison of coequalizers, $\SI_{\GA}$ is the object function of a  functor  $\GA[\sT] \rtarr \sS$.   A generalization of the following result is proven in \cite[Proposition 1.3 and Remark 1.25]{KMZ1}.  The proof there shows that $\et_{\GA}$ is a map of $\GA$-algebras.

\begin{prop}\label{keyadj} The functor $\SI_{\GA}\colon \GA[\sT] \rtarr \sS$ is left adjoint to $\OM_{\GA}$.  Moreover, the unit $\et_{\GA}$ and counit  $\epz_{\GA}$  of the adjunction are described by the following diagrams. 
$$ 
\xymatrix{  
Y \ar[r]^-{\et}  \ar[dr]_{\et_{\GA}} & \OM\SI Y \ar[d]^{\OM q}   &   \text{and}  &      \SI {\GA} \OM Z  \ar@<.7ex>[rr]^{\epz}  \ar@<-.7ex>[rr]_-{\SI \OM\epz}  &&  \SI  \OM Z \ar[rr]^-{q}  \ar[d]^{\epz} & & \SI_{\GA} \OM_{\GA} Z  \ar[dll]^{\epz_{\GA}} \\
& \OM_{\GA}\SI_{\GA} Y                                                               &         &             & &   Z   & & \\}
$$
\end{prop} 

A quick diagram chase gives the following consequence. 
\begin{cor}  Let $\GA_{\GA} = \OM_{\GA}\SI_{\GA}$ be the resulting adjunction monad in $\GA[\sT]$ and let $\bU_{\GA} \colon \GA[\sT]\rtarr \sT$ be the forgetful functor. Then $\OM q\colon \GA = \OM\SI  \rtarr \OM_{\GA}\SI_{\GA}= \bU_{\GA}   \GA_{\GA}$ is a map of monads, meaning that the evident unit and product diagrams commute when viewed as maps in $\sT$.
\end{cor}

Since $\OM = \bU_{\GA} \OM_{\GA}$ and $\GA$ is the left adjoint to $\bU_{\GA}$, for $X\in \sT$ and $Z\in \sS$ we have
$$\sS(\SI X,Z)  \iso \sT(X, \bU_{\GA} \OM_{\GA}Z) \iso \bG[\sT](\GA X,\OM_{\GA}Z) 
\iso \sS(\SI_{\GA} \GA X,Z).$$
Inspection of the adjunction and a triangle identity give the following conclusion.

\begin{prop}\label{formalBPQ}  For $X$ in $\sT$,  the map
$$\SI_{\GA}\epz_{\GA} \colon \SI_{\GA}\GA X = \SI_{\GA}\OM_{\GA} \SI X \rtarr \SI X$$
is an isomorphism with inverse $\xymatrix{\SI X \ar[r]^-{\SI \et} & \SI \GA X \ar[r]^-{q} & \SI_{\GA}\GA X}$. 
\end{prop}

\begin{rem} For $Z\in \sS$, a triangle identity gives that the composite 
$$\xymatrix{\OM_{\GA}Z \ar[r]^-{\et_{\GA}} & \OM_{\GA}\SI_{\GA}\OM_{\GA} Z \ar[r]^-{\OM_{\GA}\epz_{\GA}} & \OM_{\GA} Z\\}$$ 
is the identity. We conclude from \autoref{formalBPQ} that both maps are isomorphisms when $Z = \SI_{\GA} Y$ for $Y\in \GA[\sT]$. 
\end{rem} 

\subsection{The categorical monadicity theorem}\label{Beck}

If we omitted the adjective adjoint, the following definition would be standard. 

\begin{defn}  The adjunction $(\SI,\OM)$ is said to be {\it monadic} if $\OM_{\GA}$ is an adjoint equivalence of categories.  
\end{defn}

From the diagram defining $\epz_{\GA}$ in \autoref{keyadj}, it is clear that what is needed is a categorical hypothesis ensuring that the downward arrow $\epz$ is itself the coequalizer of $\epz$ and $\SI\OM\epz$, so that
$\epz_{\GA}\colon \SI_{\GA}\OM_{\GA} Z \rtarr Z$ is an isomorphism.  We require some preliminary definitions to state the hypothesis and phrase the Beck monadicity theorem in its usual form, but the real point of these definitions is precisely to ensure that isomorphism.  However,  these definitions do not require an adjunction, only a functor $\OM \colon \sS \rtarr \sT$.   Some of this, but not the key isomorphism part, is subsumed by our requirement that $\sT$ and $\sS$ are cocomplete, so that we already have all coequalizers.

Recall that a split coequalizer is given by a diagram
\begin{equation}\label{split1}
\xymatrix@1{
A  \ar@<2ex>[rr]^{f}  \ar@<-1ex>[rr]^{g}  &&  B \ar[rr]^-{q} \ar@<+3ex>[ll]^j& & C \ar@<+2ex>[ll]^i. \\}  
\end{equation}
such that
\[  qf = qg, \, \, qi = \id, \, \, fj=iq, \, \, \text{and} \, \, gj = \id. \]
It is a coequalizer since if $h\colon B\rtarr D$ is such that $hf=hg$, then $(hi)q = h$ and $hi$ is the unique map $C\rtarr D$ with this property.
If we apply any functor $\sT\rtarr \sV$ to a split coequalizer we get another split coequalizer.

\begin{defn}  An {\em $\OM$-split pair} is a pair of maps 
\begin{equation}\label{splitone}   \xymatrix@1{ A  \ar@<.5ex>[rr]^{f}  \ar@<-1ex>[rr]_{g}  &&  B \\}
\end{equation}
in $\sS$ together with a split coequalizer diagram in $\sT$: 
\begin{equation}\label{splittwo}
\xymatrix@1{ \OM A  \ar@<2ex>[rr]^{\OM f}  \ar@<-1ex>[rr]^{\OM g}  &&  \OM B \ar[rr]^-{r} \ar@<+3ex>[ll]^j& &  X \ar@<+2ex>[ll]^i. \\}  
\end{equation}
\end{defn}

Given the pair of maps $(f,g)$ in \autoref{splitone}, a {\em fork} is a map $q\colon B\rtarr C$ in $\sS$ such that $q\com f = q\com g$: 
\begin{equation}\label{splitthree}   
\xymatrix@1{ A  \ar@<.5ex>[rr]^{f}  \ar@<-1ex>[rr]_{g}  &&  B  \ar[r]^q &  C\\}
\end{equation}
If $(f,g)$ is an $\OM$-split pair and $q\colon B\rtarr C$ is a fork, then $\OM q \com \OM f = \OM q \com g$.  Setting 
$$  \et = \OM q\com i \colon   X  \rtarr \OM B\rtarr \OM C,\\$$
it follows that $\et$ is the unique map such that $\et \com r = \OM q$. 

\begin{defn} The functor $\OM\colon  \sS \rtarr \sT$ {\em creates coequalizers of $\OM$-split pairs} if 
for every  $\OM$-split pair $(f,g)$, there is a fork \autoref{splitthree}  in $\sS$ such that

\noindent
(i) The map $\et = \OM q \com i\colon X \rtarr \OM C$  such that $\et \com r = \OM q$ is an isomorphism, and

\noindent
(ii) $C$ is a coequalizer of $f$ and $g$ for any such fork  \autoref{splitthree}. 

\noindent
The functor $\OM$ {\em strictly creates coequalizers of $\OM$-split pairs} if every split coequalizer
\autoref{splittwo} lifts uniquely to a coequalizer \autoref{splitthree}.
\end{defn}

\begin{thm}\label{Beckthm}[Beck monadicity theorem]  The adjunction  $(\SI,\OM)$ is monadic  if and only if the functor  $\OM$ creates coequalizers of $\OM$-split pairs.  
\end{thm}

It is not emphasized, or even mentioned, in most standard categorical references that the proof actually gives an adjoint equivalence 
$(\SI_{\GA}, \OM_{\GA})$, and that fact is the real starting point of our homotopical analog and especially its variants in \cite{KMZ1}.

The first step in the proof is the following easy categorical observation \cite[Proposition 5.4.9]{Riehl2}.  It
has nothing to do with the given adjunction $(\SI,\OM)$.    

\begin{prop}\label{Ccreate}  For any monad $\bC$ on any category $\sT$, the forgetful functor $\bU\colon \bC[\sT]\rtarr \sT$ strictly creates coequalizers of $\bU$-split pairs.
\end{prop}

As in \cite[Corollary 5.4.10(i)]{Riehl2},  the definition of monadic implies the following consequence, which gives the ``only if''  implication of \autoref{Beckthm}. 

\begin{cor}  If $(\SI,\OM)$ is monadic, then $\OM$ creates coequalizers of $\OM$-split pairs.
\end{cor}

We complete the proof of \autoref{Beckthm} with an elaboration of the ``if'' part.

\begin{thm}\label{Beckthm2}[Monadicity] If the functor  $\OM$ creates coequalizers of $\OM$-split pairs, then
$\OM_{\GA}$ has a left adjoint $\SI_{\GA}$, and $(\SI_{\GA},\OM_{\GA})$ is an adjoint equivalence of categories.
\end{thm}
\begin{proof}
 Let  $(Y,\tha)$ be a $\GA$-algebra. Specializing \autoref{split1}, we have a split coequalizer
\begin{equation}\label{split22}
\xymatrix@1{
\GA\GA Y  \ar@<2ex>[rr]^{\GA\tha}  \ar@<-1ex>[rr]^{\mu}  &&  \GA Y \ar[rr]^-{\tha} \ar@<+3ex>[ll]^{\et}& & Y \ar@<+2ex>[ll]^{\et}. \\}  
\end{equation}
By our hypothesis or our completeness assumption, we have the coequalizer \autoref{split4}.  Applying $\OM$ to its pair gives the underlying pair of \autoref{split22}.  Moreover, by hypothesis, the map 
$$\et_{\GA} = \OM q\com \et\colon Y\rtarr \OM \SI Y \rtarr \OM \SI_{\GA}Y $$
such that $\et_{\GA}\com \tha = \OM q$ is an isomorphism.   A key diagram in the proof of a generalization of \autoref{keyadj}, namely (1.20) in the proof of Proposition 1.3 in \cite{KMZ1}, specializes to show that the following diagram commutes.   
\[
\xymatrix{
\GA Y \ar[r]^-{\GA\et_{\GA}} \ar[d]_{\tha} & \GA\OM_{\GA} \SI_{\GA}Y \ar[d]^{\OM\epz} \\
Y \ar[r]_-{\et_{\GA}} & \OM_{\GA} \SI_{\GA}Y\\}
\]
For any object $Z$ of $\sS$, we have a naturality fork
\begin{equation}\label{splitZ}  
 \xymatrix@1{ \SI\OM\SI\OM Z   \ar@<.5ex>[rr]^{\SI\OM \epz}  \ar@<-1ex>[rr]_{\epz}  &&  \SI\OM Z \ar[rr]^-{\epz} && Z.\\} 
 \end{equation}
Applying $\OM$ to this fork, we obtain the fork in the split coequalizer
\begin{equation}\label{split2again}
\xymatrix@1{ \OM\SI\OM\SI\OM Z  \ar@<2ex>[rr]^{\OM \SI\OM \epz}  \ar@<-1ex>[rr]^{\OM \epz}  &&  \OM\SI\OM Z \ar[rr]^-{\OM\epz} \ar@<+3ex>[ll]^{\et}
& &  \OM Z \ar@<+2ex>[ll]^{\et}. \\}  
\end{equation}
This is the special case $(Y,\tha) = (\OM Z,\OM\epz)$  of
\autoref{split22}, and
the content of the theorem is that, up to isomorphism, this special case is the
general case.   Our hypothesis gives that the fork \autoref{splitZ} displays a
coequalizer.  The universal property  of \autoref{split4} gives the inverse isomorphism to 
$\epz_{\GA}\colon \SI_{\GA}\OM_{\GA}Z \rtarr Z$.
\end{proof}

\section{The homotopical monadicity theorem}\label{SecHom}
\subsection{Homotopical assumptions and simplicial preliminaries}
We assume throughout that $\sT$ and $\sS$, in addition to being cocomplete, have standard notions of homotopy and have classes of weak equivalences satisfying the two out of three property.  We say that a map $f$  in $\GA[\sT]$ is a weak equivalence if it is a weak equivalence when considered as a map in $\sT$.   We assume that the functors $\OM$, $\SI$, and hence also $\GA$ and $\OM_{\GA}$  preserve weak equivalences (at least when restricted to good objects).   We do not assume that $\SI_{\GA}$ preserves weak equivalences, and in fact, despite our categorical motivation, $\SI_{\GA}$ will play no role at all in this section.  

\begin{rem} We will return to an analog $\epz_{\bC}$ defined using more structured monads $\bC$ in \cite{KMZ1}.
There, modulo a conjectured mild generalization of a theorem of Dwyer and Kan, we describe when equivalences of homotopy categories such as that given in \autoref{Onew2} below extend to equivalences of the $\infty$-categories associated to appropriate subcategories of $\bC[\sT]$ and $\sS$. 
\end{rem} 

Our proof of the homotopical monadicity theorem focuses on derived homotopy theory, defined in terms of a bar resolution  of algebras over monads.  That is a special case of the two-sided monadic bar construction, as is the homotopically well-behaved variant of the left adjoint to $\OM_{\GA}$ that will be used in the proof.
These are constructed simplicially, and we here give assumptions on simplicial objects in $\sT$ and $\sS$ that suffice for the construction.  

\begin{ass}\label{ass6} We write $s\sV$ for the category of simplicial objects in a category $\sV$. For a functor $F$, let us write $F_*$ for the functor on simplicial objects given by applying $F$ objectwise on $q$-simplices, and similarly for natural transformations. 
We assume the following results about simplicial objects in $\sT$ and $\sS$.
\begin{enumerate}[(i)]
\item  There are realization functors $\bT\colon s\sT\rtarr \sT$ and $\bT\colon s\sS\rtarr \sS$ and they are left adjoints.
\item  The functors $\bT$  on  $s\sT$ and $s\sS$ preserve homotopies.
\item  The functors $\bT$  on  $s\sT$ and $s\sS$ preserve weak equivalences between Reedy cofibrant objects.
\item  Realization commutes with the (left adjoint) functor $\SI$. That is,  for $K_*$ in $s\sT$, there is a natural isomorphism
$\xymatrix@1{  \SI \bT K_* \iso \bT \SI_* K_*}$.
\item   For $K_*\in s\sS$,  the natural map $\ga\colon \bT \OM_* K_*\rtarr \OM \bT K_*$, which is the adjoint of 
 $$ \xymatrix@1{\SI \bT \OM_*K_* \iso \bT \SI_* \OM_* K_* \ar[r]^-{\bT \epz_*} & \bT K_*,\\}$$
 is a weak equivalence if $K_{*}$ is  levelwise $\OM$-connective.
\end{enumerate}
\end{ass}

The appropriate definition of an $\OM$-connective object of $\sS$ in our general context is given in \autoref{LAcon}.  In our examples, none of the conditions (i) through (iv) present any difficulty.
However the condition (v), for the (iterated)  loop space functor on spaces, was one of the hardest results in \cite{MayGeo}.  The general version is the key ingredient in our understanding of homotopical monadicity.  We discuss each of these conditions in turn.

\begin{con}\label{Kancan}  Realization functors $\bT  \colon s\sV\rtarr \sV$ can generally be defined by categorical tensor products
$$\bT K_* =  K_*\otimes_{\DE} \DE^{u}_*.$$ 
Here $\DE$ is the simplex category (so that $K_*$ is a contravariant functor  $\DE\rtarr \sV$), $\sV$ is tensored and cotensored over some ground category $\sU$, and $ \DE^{u}_*$ is a covariant simplex functor $\DE \rtarr \sU$.  Then realization $\bT$ is a left adjoint with right adjoint $\bS$ given on $X\in \sV$ by
$$  \bS(X)_q = \Hom(\DE_q^u, X),$$
where $\Hom$ is the cotensor functor  $\sU^{op}\times \sV \rtarr \sV$.   Letting $\otimes\colon \sV \times \sU \rtarr \sV$ be the tensor functor, we have the adjunction
$$ \sV(V\otimes U, W) \iso \sV(V, \Hom(U,W))$$
for $V,W\in\sV$ and $U\in \sU$.   The categorical tensor products are defined explicitly via coequalizer diagrams
$$ \xymatrix@1{\coprod_{\ph\in \DE(m,n)}  K_n \otimes \DE_m^{u}  \ar@<.7ex>[rr]^{\phi\otimes \id}  \ar@<-.7ex>[rr]_{\id\otimes \ph}  && \coprod_n K_n\otimes \DE_n^{u} \ar[rr]^-{q} & & \bT K_*\\} $$ 
On passage to adjoints, the universal property of these coequalizers directly implies the desired adjunction
$$\sV(\bT K_*, X) \iso s\sV(K_*, \bS X).$$
For example, if $\sV$ is closed symmetric monoidal, we can take $\sU = \sV$ since $\sV$ is tensored and cotensored over itself.  When $\sU=\sV$ is the category of spaces, $\bT$ is the usual geometric realization functor $s\sU\rtarr \sU$.   When $\sV$ is an appropriate category of spectra (e.g. \cite{LMS, EKMM, MMSS}) and $\sU$ is the category of based spaces, the construction gives $\bT\colon s\sV \rtarr \sV$.  For the construction to be of interest, the functor $\DE_*^u$ must be well-chosen.  Very often, as in the cases of spaces and spectra, $(\bT,\bS)$ induces an adjoint equivalence of homotopy categories. With model categorical elucidation, it is often a Quillen equivalence.
\end{con}

\begin{rem}\label{Kancan2}[Historical note]  Simplicial theory arose from an understanding of how adding degeneracy operators overcomes deficiencies of the category of simplicial complexes.  This led Kan [Kan58, Section 3] to define a realization functor from the category  $s\mathbf{Set}$ of simplicial sets to suitable categories $\sV$, such as spaces;  a more recent treatment of  this point of view is given in [Hov99, Proposition 3.1.5].  Replacing $K_*\in s\sV$ by $K_* \in s\mathbf{Set}$ in \autoref{Kancan} gives Kan's definition of realization $|-|$ as a functor  $s\mathbf{Set}  \rtarr \sV$.  This is more familiar than the variant $\bT$ in \autoref{Kancan} that is relevant to our axioms.   When $\sU$ is the category of spaces, $|-|$ is the usual geometric realization functor $s\mathbf{Set}  \rtarr \sU$, and we shall formalize the relation between $\bT$ and $|-|$ in \autoref{SecSimp}.
\end{rem}

\begin{rem}\label{Homotopy} The notion of homotopy in $s\sV$ can  vary.  There is  a standard combinatorial notion that makes sense for any $\sV$
  (e.g. \cite[Definitions 5.1]{Mayss}).  For maps $f, g\colon K_* \rtarr L_*$ in  $s\sV$, a {\em{classical simplicial homotopy}} 
  $h\colon f\htp g$ is a sequence of maps
  $h_i\colon K_q\rtarr L_{q+1}$ in $\sV$ such that $d_0h_0 = f_q$, $d_{q+1}h_q = g_q$, and certain commutation relations relating the $h_i$ to the $d_i$ and $s_i$ are satisfied. In simplicial sets, or more generally if $s\sV$ has a
  product tensor over sets and therefore over simplicial sets, this is the same as a map $K_* \times \DE^s_1
  \rtarr L_*$ in $s\sV$ \cite[Proposition 6.2]{Mayss}.   As holds very
generally in simplicial model categories (e.g. \cite[Proposition 5.4]{RSS}), we
have an isomorphism $\bT(K_*\otimes X) \iso \bT K_*\otimes |X|$  for $K_*\in s\sV$ and $X\in sSet$, so that $\bT$ preserves homotopies when these are defined using $|\DE_1^s|$.

However, $s\sV$ often has a product tensor over another category $\sU$ that itself has an interval object $I$.  For example, $\sU$ might be the category of 
 (compactly generated) spaces.  Then we have a combinatorially quite different notion of a {\em{topological simplicial homotopy}} $h\colon f\htp g$.  It is given by a map  $h\colon K_*\times I \rtarr L_*$ in $s\sV$ such that $h_q(k,0) = f(k)$ and $h_q(k,1) = g(k)$ for all $k\in K_q$.  We again often have an isomorphism $\bT(K_*\otimes X) \iso \bT K_*\otimes X$ for $K_*\in s\sV$ and $X\in \sU$, so that realization also preserves this underutilized alternative notion of homotopy.   
 
This situation is often encountered and seldom articulated.  We are
free to use both kinds of simplicial homotopies in $s\sV$ for the same $\sV$.  With $\sV = \sU = \mathbf{Space}$, we give a very simple example in \autoref{SecSimp}.  There we can replace $X$ by the constant simplicial space $c_*X$ at $X$, so that $\bT(c_*X) = X$, and use that $\bT\colon s\sU\rtarr \sU$ commutes with products. 
\end{rem}

\begin{rem}\label{assrems} We remark on (i)-(v) in order.
\begin{enumerate}[(i)]  
\item This holds in all examples where \autoref{Kancan} applies.
\item Using \autoref{Homotopy}, this holds in all examples we know of where $s\sV$ and $\sV$ have reasonable notions of homotopy.
\item  For simplicial spaces, using earlier language, this  goes back to  \cite[Theorem A.4]{MayPerm}.  See \cite[Lemma 0]{Fied} for a compendium of this and related results. There are many more recent and more general proofs using Reedy cofibrancy, for example \cite[VII.3.6]{GoJa}.  
\item  Some examples of such commutation go all the way back to \cite{KanAdj2}.  Standard space level examples are in  \cite[Proposition 12.1]{MayGeo} and in \cite{KMZ1}.
\item As said before, this assumption is substantial. For spaces and the loop functor, it is proven in \cite[Theorems 12.3 and 12.7]{MayGeo}.   For $G$-spaces, the analog is proven in \cite{CW, Haus} by fixed point reduction to the nonequivariant case.  It is shown how to pass to spectra in \cite{Rant1},  and the same argument works for $G$-spectra \cite{GM3}.  Variants of the following remark give a starting point that applies in these cases and many others, but does not apply in general. 
\end{enumerate}
\end{rem}

\begin{rem}  Assume for the moment that $(\SI,\OM)$ is the standard (suspension, loops) adjunction in spaces.  We can then use the contractibility of path spaces to construct a commutative diagram from a given simplicial space $X_*$.
\begin{equation}\label{Nuts}
\xymatrix{
 \bT(\OM_*X_*) \ar[d] \ar[r]^{\ga} & \OM\bT(X_*) \ar[d] \\
 \bT(P_*X_*) \ar[d] \ar[r] & P \bT(X_*)   \ar[d] \\
 \bT(X_*) \ar[r]_{=} & \bT(X_*). \\}
\end{equation}
The middle arrow is obviously a weak equivalence since the source and target are contractible. The right column is a fibration and its fiber, and the left column is a quasi-fibration and its fiber when $X_*$ is levelwise connected.  Thus $\ga$ is a weak equivalence by a comparison of fiber sequences.
\end{rem}
 
We will often use the following notation and observations.

\begin{notn}  For any category $\sV$, write $c_*$ for the constant simplicial object functor $\sV\rtarr s\sV$.   
\end{notn}

\begin{rem}  For all realization functors $\bT$ that come from \autoref{Kancan}, it is clear that 
$\bT\com c_* $ is (naturally isomorphic to) the identity functor.
It is hard to imagine a realization functor $\bT$ for which this does not hold, but we do not have a proof that the axioms imply it.   Therefore, as an additional mini axiom, we assume henceforward that we have an identification of $\bT c_*X$ with $X$.\footnote{We thank Ajay Srinivasan for alerting us to this ambiguity.}
\end{rem}

\begin{lem}\label{triv}
For a simplicial object $K_*$ in a category $\sT$ and an object $J$ of $\sT$, a map $f\colon K_0 \rtarr J$ such that $f\circ d_0 = f\circ d_1$ induces a map 
$\xi_*(f) = \xi_*\colon K_*\rtarr c_* J$ and thus induces a map $\xi\colon \bT(K_*) \rtarr J$ in $\sT$.
\end{lem}

\subsection{The two-sided bar construction and the derived $\overline{Y}$}\label{Ybar} The monadic two-sided bar construction was introduced in \cite{MayGeo}. We have the monad $(\GA, \mu, \et)$, hence the $\GA$-functor $(\SI, \epz)$ with action $\epz\colon \SI \GA \rtarr \SI$, and we consider $\GA$ as a  $\GA$-functor with action $\mu = \OM \epz$.  For a $\GA$-algebra $(Y,\tha)$, we have the 
two-sided bar constructions 
\begin{equation}\label{guys}
\overline{Y} = B(\GA,\GA, Y) \ \  \text{and} \ \ \bE  Y = B(\SI, \GA, Y),
\end{equation}
the first in $\sT$ and the second in $\sS$. These are realizations of simplicial objects $B_*(\GA,\GA,Y)$ with $q$-simplices  $\GA^{q+1}Y$ and $B_*(\SI, \GA, Y)$ with $q$-simplices $\SI \GA^q Y$.  The $0$th face, $i$th face for 
$1\leq i \leq q-1$, and $q$th face are given by the action of $\GA$, by $\mu$ in successive positions, and by $\tha$.  The degeneracies are given by $\et\colon \Id\rtarr \GA$.  We sometimes abbreviate notation by writing 
\begin{equation}\label{guys*}
\overline{Y}_{*} = B_*(\GA,\GA, Y) \ \  \text{and} \ \ \bE_* Y = B_*(\SI, \GA, Y).\\
\end{equation}

We make the blanket assumption that our simplicial bar constructions are Reedy cofibrant, so that \autoref{ass6}(iii) applies. However, that is not needed for the following lemma, which instead uses \autoref{ass6}(ii).  It is proven by a standard ``extra degeneracy" argument (e.g. \cite[Proposition 9.8]{MayGeo}), using \autoref{triv} and the first notion of simplicial homotopy in \autoref{Homotopy}. 

\begin{lem}\label{extra}   For $Y\in \GA[\sT]$,  the maps  $\mu = \OM \epz\colon \GA^{q+1} \rtarr \GA^q$ induce a natural map $\ze\colon \overline{Y} \rtarr Y$ in $\sT$. It is a homotopy equivalence with homotopy inverse  $\nu\colon Y \rtarr \overline{Y}$ induced by the maps  $\et\colon Y \rtarr \GA^{q+1} Y$ .
\end{lem}
\begin{rem} Clearly $\et$ and $\nu$ are only maps in $\sT$, not in $\GA[\sT]$.  While the maps $\GA^{q+1} Y \rtarr Y$ are maps of $\GA$-algebras, we have not defined a $\GA$-algebra structure on $B(\GA,\GA, Y)$ since $\OM$ and therefore $\GA$ do not generally commute with $\bT$.   An alternative line of development might start model theoretically and replace our $\overline{Y} $ with a cofibrant replacement of $Y$ in a model category of $\GA$-algebras.  Instead, we shall elaborate what structure $\overline Y$ does have in \autoref{OPLAX}. 
\end{rem}

\subsection{$\OM$-connectivity}\label{Zbar}

\begin{notn}\label{notnbarZ} 
Write $\hat{\GA} = \SI\OM$ and let $Z\in\sS$.  Observe that $\SI \GA^q \OM Z = \hat{\GA}^{q+1}Z$ and that all faces of the simplicial object $B_*(\SI,\GA,\OM Z)$ are given by $\epz$.  This is ``dual'' to $\overline{Y}$, and we define 
\begin{equation}\label{guys2}
\widehat{Z} = B(\SI,\GA, \OM Z) = \bE\OM_{\GA} Z.
\end{equation}
We sometimes abbreviate notation by writing
\begin{equation}\label{guys2*}
\widehat{Z}_* = B_*(\SI,\GA, \OM Z) = \bE_*\OM_{\GA} Z.
\end{equation} 
\end{notn}
Application of \autoref{triv}  leads to the following analog of \autoref{extra}.

\begin{defn} \label{LAcon} For $Z\in \sS$,  the maps $\epz\colon \hat{\GA}^{q+1}  \rtarr \hat{\GA}^q$ induce a natural map $\xi\colon \widehat{Z}\rtarr Z$ in $\sS$. We say that $Z$ is $\OM$-connective if $\xi$ is a weak equivalence.  We let $\sS_{c}$ denote the full subcategory of $\OM$-connective objects in $\sS$.
\end{defn}

\begin{lem}\label{LAcon2} For any object $X$ of $\sT$, $\SI X$ is an $\OM$-connective object of $\sS$.  
\end{lem}
\begin{proof} Since $\widehat{\SI X} = B(\SI,\GA,\OM\SI X) = B(\SI,\GA,\GA X)$, $\xi\colon \widehat{\SI X} \rtarr \SI X$ is an equivalence by another application of the extra degeneracy argument, with the extra degeneracy now appearing in the right variable rather than the left variable.
\end{proof}

The following result has a slightly sketchy proof, but it applies in all examples we know.

\begin{lem}\label{LAcon3}  Realization on $\sS$ takes levelwise $\OM$-connective objects to $\OM$-connective objects.  Therefore, by \autoref{LAcon2},  $\widehat{Z}$ is itself $\OM$-connective for all $Z\in \sS$. 
\end{lem}
\begin{proof} Let $Z_*\in \sS$ with each $Z_q$ $\OM$-connective. Thus each
$B(\SI,\GA,\OM Z_q) \rtarr Z_q$
is a weak equivalence.   We have made the blanket assumption that our simplicial objects are Reedy cofibrant.  The folklore result (see \cite[Proposition B.1]{BF} and Lemma 3.8 below)  that realization of bisimplicial objects can be obtained by first realizing vertically and then horizontally or first horizontally and then vertically shows that
$$ \bT B(\SI,\GA, \OM_*Z_*)\iso B(\SI,\GA, \bT \OM_*Z_*).$$
Applying $\ga$ to the third variable and using Assumptions A(iii) and A(v), we have a weak equivalence
$$ B(\SI,\GA, \bT \OM_*Z_*) \rtarr B(\SI,\GA, \OM \bT Z_*).$$
Composing, there results a commutative diagram with top arrow a weak equivalence
$$
\xymatrix{
\bT B(\SI,\GA,\OM_*Z_*) \ar[dr]_{\bT \xi_*} \ar[rr]  & &  B(\SI,\GA,\OM \bT Z_*) \ar[dl]^{\xi} \\
& \bT Z_*  & \\}
$$
The left arrow $\bT \xi_*$ is the realization of a levelwise weak equivalence and is thus a weak equivalence by \autoref{ass6}(iii).   By two out of three,  the right arrow $\xi$ is therefore also a weak equivalence. 
\end{proof}

\begin{prop}\label{LAcon4} If $Z$ and $Z'$ are $\OM$-connective, then  a map $f\colon Z \rtarr Z'$ in $\sS$ is a weak equivalence if and only if $\OM f\colon \OM Z \rtarr \OM Z'$ is a weak equivalence in $\sT$. 
\end{prop}
\begin{proof}   Only if holds since $\OM$ is assumed to preserve weak equivalences.  Conversely, assume that $\OM f$ is a weak equivalence.  Then $\bE \OM f$ is a weak equivalence by \autoref{ass6}(iii).   Since $f \com \xi = \xi\com \bE\OM f$ by naturality and both maps $\xi$ are weak equivalences since $Z$ and $Z'$ are $\OM$-connective, $f$ is a weak equivalence by the two out of three property.  
\end{proof}

\begin{rem}\label{LAcon5}   For $Y\in \GA[\sT]$, we think of $\overline{Y}$ as a derived approximation of $Y$.  For $Z\in \sS$, we think of $\widehat{Z}$ as a derived connective cover of $Z$.   They are related by the following diagram, the commutativity of which is checked by use of \autoref{triv}.   
$$
\xymatrix{ \overline{\OM Z} = B(\GA,\GA,\OM Z) \ar[rr]^-{\ga} \ar[dr]_{\ze} & &  \OM B(\SI,\GA,\OM Z) = \OM\widehat{Z} \ar[dl]^{\OM\xi} \\
& \OM Z & \\}
$$
Here $\ze$ is always an equivalence, hence $\ga$ is a weak equivalence if and only if  $\OM\xi$ is a weak equivalence.  Since that is true if and only if $Z$ is $\OM$-connective, this shows the necessity of an assumption concerning $\OM$-connectivity in \autoref{ass6}(v). 
\end{rem}

The following two examples may make \autoref{LAcon} look more intuitive than it does at first sight.

\begin{exmp}[Dold-Kan correspondence]  Let $\SI$ be the normalized chain complex functor $sAb \rtarr Ch$ from simplicial abelian groups to (normalized) $\bZ$-chain complexes.  Let $\DE^{ch}\colon \DE \rtarr Ch$ denote the chain complex functor given by the classical simplicial complexes $\DE_n$  and let $\OM\colon Ch\rtarr sAb$  send a chain complex $C_*$ to the simplicial Abelian group with $q$-simplices the chain maps $\DE^{ch}_q \rtarr C_*$ under sums.  Then $\SI$ takes values in nonnegatively graded chain complexes and we see by unravelling definitions that $\xi$ is a quasi-isomorphism (= weak equivalence) on nonnegatively graded chain complexes, as was already proven in  \cite{KanAdj2}.
\end{exmp}

\begin{exmp}[spectra]\label{specon} Using a category of spectra with a  good zeroth space adjunction $(\SI^{\infty},\OM^{\infty})$,  $B(\SI^{\infty}, \GA, \OM^{\infty} Z)$ is a connective cover of a spectrum $Z$, so that $\xi$ is a weak equivalence if and only if the spectrum $Z$ is connective in the classical sense of having homotopy groups zero in negative degrees \cite{KMZ1}. 
\end{exmp}

\subsection{The crude homotopical monadicity theorem}\label{Crude}

\begin{thm}[Crude Homotopical Monadicity]\label{HomMon} Under our assumptions, the functors $\bE\colon \GA[\sT] \rtarr \sS_{c}$ and $\OM_{\GA}\colon \sS_{c} \rtarr \GA[\sT]$ preserve weak
equivalences.  For $Z\in \sS_{c}$,  $\xi\colon \bE \OM_{\GA} Z = \widehat{Z}\rtarr Z$
is a natural weak equivalence.  For $Y\in \GA[\sT]$, we have natural weak equivalences in $\sT$
$$ \xymatrix@1{ Y & \overline{Y} \ar[l]_-{\ze}  \ar[r]^-{\ga} & \OM\bE Y. \\}$$
\end{thm}
\begin{proof}   $\bE$  preserves weak equivalences by \autoref{ass6}(iii) and $\OM$ is also assumed to preserve weak equivalences. The weak equivalence $\ze$ is given by \autoref{extra} and the weak equivalences $\xi$ and $\ga$ are given by \autoref{LAcon} and \autoref{ass6}(v).  That compares the identity to $\OM \bE$. Applying $\OM$, which is sufficient by \autoref{LAcon4}, the diagram in \autoref{LAcon5} compares $\bE\OM$ to the identity.   \end{proof}

 There is a glaring gap.  We would like to say that $\overline{Y}$ is a $\GA$-algebra such that $\ze$ and $\ga$ are maps of $\GA$-algebras.   Axioms for a more highly structured monad $\bC$ that acts on $\OM Z$ for $Z\in \sS$ ensure that this holds  with $\GA$ replaced by $\bC$.   That is the subject of \cite{KMZ1}.   As we explain next, something close to that is true here.
 
 \subsection{Linked $\GA$-algebras and the sharpened homotopical monadicity theorem}\label{OPLAX}
 We describe what structure $\overline{Y}$ has in our present context  and describe how it links the $\GA$-algebras $Y$ and $\OM \bE Y$.  This is a specialization of abstract theory given in \autoref{ablink}.
  
We have adjoint functors $(\SI_*,\OM_*)$, and their adjunction monad $\GA_*$ on
$s\sT$.  We also have the realization functor $\bT\colon s\sT\rtarr \sT$.
(The following definitions will not use the realization functor $\bT\colon s\sS\rtarr \sS$.) Continue to write $\ga$ for its specialization 
$$\ga\colon  \bT \GA_* =  \bT(\OM_* \SI_* (-)) \rtarr \OM \bT(\SI_*(-)) \iso \OM \SI\bT= \GA\bT.$$
The definition of an op-lax map of monads is recalled in \autoref{oplax}.

\begin{prop}\label{oplaxprop}  $(\bT,\ga)$ is an op-lax map of monads $\GA_*\rtarr \GA$.
\end{prop}
\begin{proof}
We must show that the following two diagrams commute.
\begin{equation}\label{oplax2a}
\xymatrix{
  & \bT \ar[dl]_{\bT\et_*} \ar[dr]^{\et\bT}  &  & & \bT \GA_*\GA_* \ar[r]^-{\ga \GA_*}  \ar[d]_{\bT \mu_*} &  \GA \bT\GA_* \ar[r]^-{\GA \ga} &  \GA\GA \bT  \ar[d]^{\mu\bT} \\
 \bT\GA_* \ar[rr]_-{\ga}  & & \GA \bT  & & \bT \GA_*  \ar[rr]_-{\ga} & & \GA \bT \\ }
\end{equation} 
Passing to adjoints and using the definition of $\ga\colon |\OM_* Y_*| \rtarr \OM |Y_*|$ in terms of $\epz_*$ together with a triangle identity, it is easy to check that the unit diagram commutes.  The product diagram is obtained from the
product diagram of the following result by restricting its variables to be objects $\SI_* X_*$ for $X_*\in s\sT$ rather than general objects $Y_* \in s\sS$.  
\end{proof}
Using  $\epz$, we have seen that the functor $\OM\colon \sS\rtarr \sT$ takes values in $\GA$-algebras.  Analogously, the functor $\OM_*\colon s\sS \rtarr s\sT$ takes values in $\GA_*$-algebras.  For simplicity of notation, we write $\OM$ and $\OM_*$ instead of the more logical $\OM_{\GA}$ and $\OM_{\GA_{*}}$ in the rest of this subsection.

\begin{prop}\label{oplaxprop2}  The op-lax map of monads $(\bT,\ga)\colon (\GA_*,\OM_*) \rtarr
  (\GA,\OM)$ links the functors $\OM_*$ and $\OM$ in the sense that the
  following diagrams commute.
 \begin{equation}\label{oplax2}
\xymatrix{
\bT\OM_* \ar[d]_{\bT\et_*\OM_*} \ar[dr]^{\et\bT\OM_*} \ar[rr]^-{\ga}  & &\OM \bT \ar[d]^{\et\OM\bT}   & & \bT \GA_*\OM_* \ar[r]^-{\ga \OM_*}  \ar[d]_{\bT \OM_*\epz_*} &  \GA \bT\OM_* \ar[r]^-{\GA \ga} &  \GA\OM \bT  \ar[d]^{\OM\epz\bT} \\ 
 \bT \GA_*\OM_* \ar[r]_{\ga\OM_*} & \GA\bT \OM_* \ar[r]_-{\GA\ga} & \GA\OM\bT  & & \bT \OM_*  \ar[rr]_-{\ga} & & \OM \bT  \\ }
\end{equation} 
\end{prop}
\begin{proof}  The unit diagram is redundant since its triangle commutes by the unit diagram in \autoref{oplaxprop}, and its square commutes by the naturality of $\et$. We see that the product diagram commutes by passing to adjoints and expanding slightly, using the original adjoint specification of $\ga$.
\begin{equation}\label{oplax3}
\xymatrix{
\SI\bT \GA_*\OM_*  \ar[d]_{\SI\bT \OM_*\epz_*} \ar[r]^-{\iso}  & \bT\SI_*\OM_*\SI_*\OM_* \ar[r]^-{\bT\epz_*} \ar[d]_{\bT\SI_*\OM_* \epz_*}
& \bT\SI_*\OM_*   \ar[d]_{\bT\epz_*}&  \SI\bT\OM_* \ar[l]_-{\iso} \ar[dl]_{\tilde{\ga}} \ar[dr]^{\tilde{\ga}} \ar[r]^{\SI\ga} & \SI\OM \bT \ar[d]^{\epz\bT}\\
\SI \bT \OM_* \ar[r]_-{\iso} & \bT \SI_*\OM_* \ar[r]_-{\bT\epz_*}  & \bT  \ar[rr]_-{=} &   & \bT \\ }
\end{equation} 
The two left squares and the lower triangle commute trivially; $\tilde{\ga}$ denotes the adjoint of $\ga$, the upper left triangle commutes by our definition of $\ga$ and the upper right triangle commutes by the definition \autoref{adj3} of adjoint maps.
\end{proof}

A generalization of the following definition to algebras over any two monads connected by an oplax map is given in \autoref{CDlink}.

\begin{defn}\label{link}  Let $(Y_*,\tha_*)$ be a simplicial $\GA_*$-algebra and $(\bT Y_*, \tha)$ be a $\GA$-algebra.  We say that these objects are linked if the following diagrams commute. 
\begin{equation}
\label{linkeq}
\xymatrix{
\bT Y_{*} \ar[d]_{\bT \et_{*}} \ar[dr]^{\et\bT} &     & & \bT \GA_*Y_* \ar[r]^-{\ga}  \ar[d]_{\bT \tha_*} & \GA \bT Y_* \ar[d]^{\tha} \\
\bT \GA_*Y_* \ar[r]_{\ga} & \GA \bT Y_*              & & \bT Y_*  \ar[r]_-{=} & \bT Y_* \\ }
\end{equation}
To see the idea, combine these diagrams with the diagrams of \autoref{oplax2} and
the defining unit and product diagrams of the given algebras.   These diagrams show that $\ga$ maps 
$\bT$ applied to the simplicial $\GA_*$-algebra $(Y_*,\tha_*)$ to the $\GA$-algebra $(\bT Y_*, \tha)$. 
When we have the diagrams \autoref{linkeq}, we call $\ga$ a linking map.
\end{defn}

However, sensible though it is, our main examples are not quite of the form given by \autoref{link}.  Instead, take $Y_*$ in \autoref{link} to be $\OM_* Z_*$ for a simplicial object $Z_*\in s\sS$,
  so that $(\OM_* Z_*,\tha_*)$ is a simplicial $\GA_*$-algebra. It is not
  guaranteed that $(\bT \OM_* Z_*, \tha)$ is a $\GA$-algebra, but we have the map
  $\gamma: \bT \OM_{*}Z_{*} \rtarr \OM \bT Z_{*}$ to the $\GA$-algebra $ \OM \bT Z_{*}$. Here we don't have the right-hand  action map $\tha$ in \autoref{linkeq}, but we may compose 
  $\ga$ with $\GA\ga$ and  
apply the product diagram of \autoref{oplax2} to $Z_{*}$.  This leads to a variant of the notion of a linked pair of objects that applies in
  our context.
\begin{defn}\label{link2} We say that the simplicial $\GA_*$-algebra 
  $\OM_* Z_*$ and the $\GA$-algebra $\OM\bT Z_*$ are linked via the linking map $\ga\colon \bT\OM_*Z_* \rtarr \OM \bT Z_*$. The actions are
  given by
  $$\OM_*\epz_*\colon \GA_*\OM_* Z_* \rtarr \OM_*Z_* \text{ and } \OM \epz \bT\colon \GA \OM
  \bT Z_* \rtarr \OM \bT Z_*.$$
  We view \autoref{oplax2} as unit and product linking diagrams and we interpret
  \autoref{oplaxprop2} as saying that these objects are functorially linked for
  all $Z_*\in s\sS$. The examples of interest are of the form $B_*(\SI, \GA, Y)
  \equiv \bE_*Y$ for a $\GA$-algebra $Y$.  Here $\OM_* \bE_*Y$ is the $\GA_*$-algebra
  $B_*(\GA, \GA, Y) \equiv \overline{Y}_*$.
\end{defn}

We do have a trivial example that fits into \autoref{link} and relates to the linked pairs of \autoref{link2}.

\begin{exmp}\label{exmp1}  If $Y$ is a $\GA$-algebra, then $c_*Y$ is a $\GA_*$-algebra such that $\bT c_*Y = Y$ as $\GA$-algebras.  The map 
$$\ga\colon \bT \GA_* c_*Y \rtarr \GA \bT c_*Y = \GA Y$$
 is the identity since $\GA_* c_* =  c_*\GA $.  Thus  $(c_*Y,\tha_*)$ and $(Y,\tha)$ are linked trivially.  
\end{exmp}

Summarizing, we have the following structure on the maps $\ze\colon \overline{Y} \rtarr Y$ and $\ga\colon \overline{Y} \rtarr \OM\bE Y$.

\begin{thm}[Sharpened homotopical monadicity]\label{HomMon2}  Let $(Y,\tha)$ be a $\GA$-algebra and consider the maps 
$$ \xymatrix{
 \ar[d]_{=}Y  & \overline{Y}  \ar[l]_-{\ze} \ar[r] ^-{\ga} \ar[d]^{\equiv}  & \OM\bE Y \ar[d]^{\equiv} \\
\bT c_* Y    &  \bT \OM_*\bE_* Y  \ar[l]^{\bT\ze_*}  \ar[r]_{\ga} & \OM \bT \bE_* Y  \\}
$$
\begin{enumerate}[(i)]
\item Forgetting structures, $\ze$ is a homotopy equivalence  and $\ga$ is a weak equivalence.  
\item As in \autoref{exmp1}, $Y = \bT c_{*} Y$ and  $\ze = \bT \ze_*$,  where $\ze_*$ is a map of
  simplicial $\GA$-algebras.
\item As in \autoref{link2}, $\ga$ is a linking map relating $\overline{Y}$ to the  $\GA$-algebra $\OM \bE Y$.   
\end{enumerate}
\end{thm}

\begin{defn}\label{Onew} Consider the full subcategory of $\sT$ whose objects are the $\GA$-algebras.   Using the natural inverse equivalence $\nu\colon Y\rtarr \overline{Y}$ of $\ze$, the pair $(\ze,\ga)$ determines a {\em{linked weak equivalence}} $\ga\com \nu \colon Y \rtarr \OM \bE Y$. Inverting  these maps and their composites, rather than just inverting maps of $\GA$-algebras that are weak equivalences in $\sT$, we can construct a new  homotopy category of $\GA$-algebras.  
\end{defn}

Using \autoref{Onew}, Theorems \ref{HomMon} and  \ref{HomMon2} give the following conclusion.

\begin{cor}\label{Onew2} The homotopy category of $\GA$-algebras is equivalent to  the homotopy category of $\OM$-connective objects of $\sS$.
\end{cor}

\begin{rem}\label{moreC} Using more structured monads $\bC$ such that $\bC \bT$ is isomorphic to $\bT\bC_*$, we avoid the  use of  linked pairs in \cite{KMZ1}.  There the analogous resolution $\overline{Y}$ of a $\bC$-algebra is a $\bC$-algebra.   
\end{rem}

\section{Simplicial sets, simplicial spaces, and spaces}\label{SecSimp}
\subsection{The composite classical adjunction $(|-|,\mathbf{Sing})$}
This is a kind of musing on questions about simplicial spaces the senior author
thought a bit about when writing \cite{Mayss}.  Specialize $(\SI,\OM)$ in our general context to the classical
(realization, total singular complex) adjunction $(|-|,\mathbf{Sing})$ between simplicial sets $s{\mathbf{Set}}$  
and topological spaces $\sU$.\footnote{To reiterate, the notation $(\SI,\OM)$ is generic.  Especially in this section, one must resolutely ignore its standard usage.}
Here we are using the old notation of \autoref{Kancan2}. We shall compare that with the analogous adjunction $(\bT, \bS)$ of \autoref{Kancan}, where $\bT$ is the realization functor from simplicial spaces $s\sU$ to spaces $\sU$ and $\bS$ is defined exactly as $\mathbf{Sing}$, except that we take the $q$-simplex object $\bS_q(X)$ to be the space of continuous maps $\DE^t_q \rtarr X$ with the compact open topology.   

We also have the adjunction $(\bD,\bU)$ between $s{\mathbf{Set}}$ and $s\sU$; $\bD$ gives the set of $q$-simplices of a simplicial set its discrete topology and $\bU$ is the forgetful functor that takes the underlying set of the space of $q$-simplices.  By definition, the adjunction $(|-|, \mathbf{Sing})$ factors as $(\bT,\bS)\com (\bD,\bU)$.  That is, $|-| = \bT\bD$ and $\mathbf{Sing} = \bU\bS$.   Clearly $\et\colon \Id \rtarr \bU\bD$ is the identity but $\epz\colon \bD\bU\rtarr \Id$  is far from being an equivalence.  

Let $\GA_{s} = \mathbf{Sing}\com |-|$ and $\GA = \bS\bT$.  These are monads in $s{\mathbf{Set}}$ and in $s\sU$.  Of course, it is not to be expected that $\bU\colon s\sU \rtarr s{\mathbf{Set}}$ preserves weak equivalences.  The following diagram gives the formal categorical picture.  The larger triangle is a composite analog of the smaller one.   Remember that $|-| =\bT\bD$ and $\mathbf{Sing} = \bU\bS$.

 \begin{equation}\label{adj6} 
\xymatrix{
s{\mathbf{Set}} \ar@<.5ex>[rr]^{\bD}  \ar@<.5ex>[ddrr]^{\bF_{\GA_s}} & & s\sU  \ar@<.5ex>[ll]^{\bU}    \ar@<.5ex>[rr]^{\bT}  \ar@<.5ex>[dr]^{\bF_{\GA}}& & \sU   \ar@<.5ex>[ll]^{\bS}   \ar@<.5ex>[dl]^{\bS_{\GA}}   
\ar@/^4pc/@<.6ex> [ddll]^-{\mathbf{Sing}_{\GA_s}}        \\
& & &   \ar@<.5ex>[ul]^{\bU_{\GA}}  \GA[s\sU]   \ar@<.5ex>[dl]^{\bU}  \ar@<.5ex>[ur]^{\bT_{\GA}}& \\
& & \ar@<.5ex>[uull]^{\bU_{\GA_s}}  \GA_s[s{\mathbf{Set}}] \ar@{-->}@<.5ex>[ur]^{\bD} \ar@/_4pc/@<.6ex> [uurr]^-{|-|_{\GA_s}}& &  \\}
\end{equation}

The functor $\bU\colon \GA[s\sU] \rtarr \GA[s{\mathbf{Set}}] $ exists formally  by \cite[Proposition 5.11(i)]{KMZ1}  but also exists by inspection.  It is unclear to us whether or not the lower arrow $\bU$ has a left adjoint $\bD$, but we suspect not.  The curved adjunction exists formally by \cite[Proposition 1.3 and Lemma 5.3]{KMZ1}.

Diagrams like this, but satisfying conditions that are not satisfied here, are studied in \cite[Part II]{KMZ1}.   There the dotted arrow may or may not exist but the curved adjunction always exists.   It is standard that $\epz\colon |\mathbf{Sing} X |\rtarr X$ is a CW approximation of $X$ with $0$-cells the points of $X$.   In contrast, we have a sharper result for $\bT\bS$. We ignore weak equivalences for the moment,   just thinking of homotopy categories in terms of topological simplicial homotopy classes of
maps as defined in \autoref{Homotopy}.  As there, again let $c_*X$ denote the constant simplicial space at a space $X$.

\begin{prop}\label{new0}
For spaces $X$, $\bS X$ is naturally homotopy equivalent to $c_*X$.
\end{prop}
\begin{proof} Define $\pi\colon \bS X \rtarr c_*X$ and $\io\colon c_*X\rtarr \bS  X$ on and to $q$-simplices as follows.   For $\ph\colon \DE^t_q \rtarr X$ let $\pi(\phi) = \ph(0^q, 1)\in X$, where $(0^q,1)$ is the last vertex of $\DE^t_q$.   Since $\DE^t_q$ is contractible, its constant map to $(0^{q},1)$ is a homotopy equivalence, and these maps for varying $q$ induce the simplicial map $\pi$.  For $x\in X$, define the $q$-simplex $\io_q(x)$ to be $s_{q-1}\cdots s_0\io_0(x)$, where $\io_0(x)(1) = x$. Observe inductively that $\io_q(x)\colon \DE^t_q \rtarr X$ is constant at $x$.   Clearly $\pi\com \io = \id$.  Recall that the faces and degeneracies on $\bS_qX$ are given by
$$(d_i(\ph))(t_0, \cdots, t_{q-1}) = \ph(t_0, \cdots t_{i-1},0,t_i, \cdots, t_{q-1})$$
and
$$(s_i(\ph))(t_0,\cdots, t_{q+1}) = \ph (t_0, \cdots t_{i-1},t_i+ t_{i+1}, t_{i+2}, \cdots, t_{q+1}).$$
Define a (topological simplicial) homotopy $h\colon \bS X\times I \rtarr S X$ on $q$-simplices $\ph$ by
$$h(\ph,t)(t_0, \cdots, t_q) = \ph(t t_0, \cdots , t t_{q-1}, t t_q +1-t)$$
By quick verifications, $h(\ph,0) = (\io\com \pi)(\ph)$,  $h(\ph, 1) = \ph$, and
$h$ commutes with $d_i$ and $s_i$.  Thus $h\colon \io\com \pi \htp \id$ and in
fact $h$ shows that  $\pi$ is a deformation retraction of $\bS X$ onto $c_*X$.  
\end{proof}

\begin{thm}\label{new?} The adjunction $(\bT,\bS)$ induces an adjoint equivalence between the homotopy category of simplicial spaces and the homotopy category of spaces.
\end{thm}
\begin{proof}
Using \autoref{new0} and tracing through the definitions in its proof, we see that %Since $\epz|\ph,u| = \phi(u)$,  $\epz|\io(x), u| = \io(x) = u$.  under the identification of $\bT c_*X$ with $X$, 
$\epz\com \bT \io$ is the identity on $X$.  Composing with $\bT\pi$, it follows that $\epz$ is homotopic to $\bT\pi$ and is therefore a homotopy equivalence.  A triangle identity then gives that $\et$ is also a homotopy equivalence.
\end{proof}

\begin{rem} We should say something about model structures, but the literature is focused on simplicial model structures.  For example, Dugger \cite{Dug} gives a simplicial model structure on $s\sU$ such that the constant simplicial space functor $c_*\colon \sU \rtarr s\sU$ and the zeroth space functor $z_0\colon s\sU \rtarr \sU$ specify a Quillen equivalence.  His model structure on $\sU$ is Quillen's, and his model structure on $s\sU$ is a suitable localization of the Reedy model structure. His focus is on the fact that this is a simplicial model structure.  We have the following composite adjunction, which should be compared with that of \autoref{adj6}.
\begin{equation}\label{adj7} 
\xymatrix{
\sU \ar@<.5ex>[rr]^{c_*} & & s\sU  \ar@<.5ex>[ll]^{z_0}    \ar@<.5ex>[rr]^{\bT}  & & \sU   \ar@<.5ex>[ll]^{\bS}    \\}
\end{equation}
Both the composite left adjoint and the composite right adjoint are naturally isomorphic to the identity functor of $\sU$.  Dugger's focus is on the left adjunction and the simplicial enrichment of $s\sU$.  Our focus is on the right adjunction and the topological enrichment of $s\sU$. These enrichments are quite different, as the tensors with simplicial sets and with spaces in the first and second definitions of homotopies in \autoref{Homotopy} illustrate.  It is the topological focus that makes the proofs of \autoref{new0} and \autoref{new?} so trivial.  As Emily Riehl suggested to us, it seems likely that $(\bT,\bS)$ is a Quillen adjunction, where we use Dugger's model structure on $s\sU$.  We  believe that this should also be a Quillen equivalence with respect to a more topologically grounded model structure on $s\sU$, perhaps constructed along the lines of having weak equivalences and cofibrations created by the left adjoint $\bT$.  However, we have not pursued these ideas.
\end{rem}

\subsection{\autoref{ass6} in the contexts of simplicial sets and simplicial spaces}
The previous subsection had nothing to do with homotopical monadicity.
 We return to that now.  We take our fixed adjunction $(\SI, \OM)$ (as in \autoref{SecHom})  to be either
 $$|-|\colon s\mathbf{Set} \rtarr \sU \ \ \text{and} \ \ \mathbf{Sing}\colon \sU \rtarr s\mathbf{Set} $$
 or
 $$\bT\colon s\sU\rtarr \sU  \ \ \text{and} \ \ \bS\colon \sU \rtarr s\sU.$$
Recall that $\mathbf{Sing}$ is the total singular complex functor and that $\bS$ gives its set of $q$-simplices  the standard function space topology  for all $q$. 

To see \autoref{ass6}, we need left adjoint realization functors 
$$ \bT_{ver}^s \colon ss{\mathbf{Set}} \rtarr s{\mathbf{Set}}, \ \ \text{and}  \ \  \bT_{ver}^t\colon  ss\sU\rtarr s\sU.$$ 
 
Like $\bT$, these are given by \autoref{Kancan}, as formalized concretely in \autoref{simpAi} below. 
Note that $\bT$ in \autoref{ass6} now refers either to  $\bT_{ver}^s$ and $\bT$ or to $\bT_{ver}^t$ and $\bT$.   It may help to emphasize that we are thinking of $ss\sV$ as simplicial objects (horizontal) in $s\sV$ (vertical), so that when writing $X_{*,*}$ for any kind of object $X$, the first $*$ refers to the horizontal and the second to the vertical variable.

Recall that $s{\mathbf{Set}}$ is a closed cartesian monoidal category with product and internal hom given by \cite[Definitions 6.1 and 6.4]{Mayss}, so that it is tensored and cotensored over itself.  To apply \autoref{Kancan}, we use the covariant functor $\DE^s\colon \DE \rtarr s{\mathbf{Set}}$ such that $\DE^s_q$ is the simplicial set represented by the object $\bq\in \DE$.  

We regard $s\sU$ as tensored and cotensored over $\sU$ via levelwise cartesian product and levelwise hom  $\sU^{op}\times s\sU \rtarr s\sU$ and we apply \autoref{Kancan} using  the covariant functor $\DE^t\colon \DE \rtarr s\sU$ such that $\DE^t_q$ is the standard topological $q$-simplex.

\begin{defn}\label{simpAi}  Applying \autoref{Kancan}  with $\sV = \sU = s{\mathbf{Set}}$,  and  $\DE_*^u = \DE^s$ or with $\sV = s\sU$,  $\sU = \sU$, and  $\DE_*^u = \DE^t$, we obtain the required  left adjoint realization functors $\bT_{ver}^s$ and 
$\bT_{ver}^t$.  Thus \autoref{Kancan} directly gives \autoref{ass6}(i).  
\end{defn} 

Digressing slightly, we place this definition in a general context before returning to \autoref{ass6}.
Bisimplicial sets were treated in \cite[Appendix B]{BF} and, more recently, \cite[Chapter IV]{GoJa} and elsewhere.
In contrast, we have not found a treatment of bisimplicial spaces in the literature. In fact, bisimplicial spaces are very often defined to mean bisimplicial sets.   The following general bisimplicial analog of \autoref{Kancan} defines $\bT_{ver}^v$ in any reasonable category $\sV$.  

\begin{con}\label{Kancan3}  With notations as in \autoref{Kancan} and with $\sU$ assumed to be cartesian monoidal,  realization functors $\mathbf{Tot}  \colon ssV\rtarr V$ can generally be defined by categorical tensor products
$$ \mathbf{Tot}  K_{**} =  K_{**}\otimes_{\DE\times \DE} (\DE^{u}_*\times \DE^{u}_*).$$ 
As in \autoref{Kancan}, $\mathbf{Tot}$ is then a left adjoint with the $(p,q)$ simplices of the right adjoint given by 
$\Hom(\DE_p^u\times\DE_q^u, X)$. Here, as in \autoref{Kancan}, $\Hom$ is the cotensor functor  $\sU^{op}\times \sV \rtarr \sV$.  
The categorical tensor product here is defined explicitly via the coequalizer diagram
\begin{equation}\label{bibi}
 \xymatrix@1{\coprod_{(\ph,\ps)\in \DE(m,p)\times \DE(n,q)}  K_{p,q} \otimes (\DE_m^{u}\times \DE_n^u)  \ar@<.7ex>[r]^-{(\phi,\ps)\otimes \id}  \ar@<-.7ex>[r]_-{\id\otimes (\ph,\ps)}  & \coprod_{p,q} K_{p,q}\otimes \DE_p^{u}\times\DE_q^u \ar[r] &  \mathbf{Tot}K_{**}.\\}
 \end{equation} 
For each fixed $p$, we have an object $K_{p,\ast}\in s\sV$ that we  can realize to obtain a levelwise vertical realization 
$\bT_{ver}\colon ss\sV\rtarr s\sV$.  Analogously, fixing $q$ and realizing, we can obtain a levelwise horizontal realization 
$\bT_{hor}\colon ss\sV\rtarr s\sV$.  We add a superscript $v$ when $\sV$ is ambiguous. Breaking the domain coproduct in \autoref{bibi} as 
$\coprod_{\ph\colon m\rtarr p}\coprod_{\psi\colon n\rtarr q}$ and omitting some evident requisite categorical hypotheses, we find that \autoref{bibi} is isomorphic to the composite diagram
{\small{
\begin{equation}\label{bibi2}
\xymatrix{
\coprod_{\ph\in\DE(m,p)} (\coprod_{\ps\in \DE(n,q)} K_{p,q}\otimes \DE_{n}^u)\otimes \DE_{m}^u
\ar@<.7ex>[r]^-{(\ps\otimes \id)\otimes \id}  \ar@<-.7ex>[r]_-{(\id \otimes \ps)\otimes\id}  & 
\coprod_{\ph\in\DE(m,p)} (\coprod_q K_{p,q}\otimes \DE_q^{u}) \otimes \DE_m^u  \ar[r] &\\
 \coprod_{\ph\in \DE(m,p)} (\bT_{ver}K_{*,*})_p \otimes\DE_m^u  
 \ar@<.7ex>[r]^-{\ph\otimes \id}  \ar@<-.7ex>[r]_-{\id\otimes \ph}  & \coprod_{p}(\bT_{ver}K_{*,*})_p \otimes \DE_p^u \ar[r] & \bT\bT_{ver} K_{\ast,\ast} \\}
\end{equation}
}}
We have a similar isomorphism with the roles of vertical and horizontal reversed.  This gives Fubini theorem type isomorphisms
\begin{equation}\label{Fubini}
\bT \bT_{ver} K_{**} \iso \mathbf{Tot} K_{**} \iso \bT \bT_{hor} K_{**}.
\end{equation}
\end{con}

\begin{rem} Although we will not use it, \cite[Proposition B.1]{BF} says that $\bT_{ver}^s$ is isomorphic to the functor that sends a bisimplicial set $K_{*,*}$ to the diagonal simplicial set $\mathbf{Diag}K_{*,*}$ with $q$-simplices $K_{q,q}$. It is folklore that this is true up to homotopy for simplicial spaces. We understand how to try to prove a more general analogue, but we have not pursued the details.
\end{rem}

The following diagrams give a picture of our framework.

\begin{equation*}
  \begin{tikzcd}
    ss\mathbf{Set} \ar[d,"|-|_2"'] \ar[r,shift left, "|-|_{*}"]
    & s\sU \ar[d,"\bT"] \ar[l, shift left, "\mathbf{Sing}_{*}"]
    &     ss\sU \ar[d,"\bT_2"'] \ar[r,shift left, "\bT_{*}"]
    & s\sU \ar[d,"\bT"] \ar[l, shift left, "\bS_{*}"]\\
    s\mathbf{Set} \ar[r,shift left, "|-|"] & \sU \ar[l, shift left, "\mathbf{Sing}"]
    &     s\sU \ar[r,shift left, "\bT"] & \sU \ar[l, shift left, "\bS"]
  \end{tikzcd}
\end{equation*}

As in \autoref{Homotopy}, \autoref{ass6}(ii) about preservation of homotopies  holds formally, where classical simplicial homotopies are used in $ss{\mathbf{Set}}$ and $s{\mathbf{Set}}$ but topological simplicial homotopies are used in $ss\sU$ and $s\sU$.

For \autoref{ass6}(iii), we need to be precise about weak equivalences.  We have choices, and we adopt the following definition.  Recall that one correct definition of a weak equivalence of simplicial sets is a map $f$ such that $|f|$ is a weak equivalence of spaces.  Analogously, we define a weak equivalence of simplicial spaces to be a map $f$ such that $\bT f$ is a weak equivalence of spaces.  A plausible alternative might be to require $f$ to be a levelwise weak equivalence. However, for Reedy cofibrant simplicial spaces, realizations of levelwise weak equivalences are weak equivalences, although that is not well-documented in the literature; see \cite{SP} for discussion. 

\begin{defn}\label{ouch} A map  $f \in ss{\mathbf{Set}}$ is a weak equivalence if $|-|$ applied levelwise gives a levelwise weak equivalence in $s\sU$.  Similarly, assuming levelwise Reedy cofibrancy, a map  $f \in ss\sU$ is a weak equivalence if $\bT$ applied levelwise gives a levelwise weak equivalence in $s\sU$.
\end{defn}

For the functors we start with, $|-| \colon s{\mathbf{Set}} \rtarr \sU$ preserves weak equivalences by definition and $\mathbf{Sing} \colon \sU \rtarr s{\mathbf{Set}}$ preserves weak equivalences by the two out of three property since $\epz\colon |\mathbf{Sing}X| \rtarr X$ is a natural weak equivalence.  For  the analogs for $\bS$ and $\bT$ we again have to take Reedy cofibrancy into account, but \autoref{new0} shows that there is no real problem. 

We now consider \autoref{ass6}(iii)-(v), starting from the domain category $s{\mathbf{Set}}$ or $s\sU$ and the  target category  $\sU$ of our adjunctions $(|-|,\mathbf{Sing})$ and $(\bT,\bS)$. 

\begin{lem}[\autoref{ass6}(iii)]  If  $f_{*,*}\colon K_{*,*}\rtarr L_{*,*}$ is a weak equivalence between  objects of $ss{\mathbf{Set}}$ or of $ss\sU$, levelwise Reedy cofibrant both vertically and horizontally  in the latter case, then its realization $f_* = \bT_{ver}^s f_{*,*}$  or $f_* = \bT_{ver}^t f_{*,*}$ is a weak equivalence in $s{\mathbf{Set}}$ or in $s\sU$.
\end{lem} 
\begin{proof} By hypothesis, $\bT_{ver}^s$ or $\bT_{ver}^t$ applied levelwise gives a levelwise weak equivalence $f_*$.  For  $ss{\mathbf{Set}}$, the result is  \cite[Theorem B.2]{BF}. For $ss\sU$, $\bT$ takes $f_*$ to a weak equivalence in $\sU$, which means that $f_*$ is a weak equivalence in $s\sU$.  
\end{proof}

\begin{lem}[\autoref{ass6}(iv)] \label{Tcommute} For $K_{*,*}\in ss{\mathbf{Set}}$ and  $X_{*,*} \in ss\sU$,   
$$  |\bT_{ver}^s K_{*,*}| \iso \bT |K_{*,*}|_*. \ \ \text{and} \ \  \bT \bT_{ver}^t X_{*,*} \iso \bT\bT_*X_{*,*}$$
\end{lem}
\begin{proof}  We must reinterpret what \autoref{ass6}(iv) is saying in our present context.  In the first isomorphism
$|K_{*,*}|_*\in s\sU$ on the right has  $p$th space the realization $|K_{p,*}|$ and this
commutation of left adjoints holds by application of the categorical Fubini's theorem, as in the version formulated in \autoref{Kancan3}.  In the second, everything is topological. Here $\bT_*X_{*,*}\in s\sU $ has $p$th  space $\bT X_{p,*}$, so that it is exactly $\bT_{ver}^t$ and the  claimed isomorphism is a tautology.
\end{proof}

\begin{prop}[\autoref{ass6}(v)]\label{all?} For $X_{*}\in s\sU$,  consider the natural map of simplicial sets 
$$ \ga\colon \bT_{ver}^s \mathbf{Sing}_* X_*   \rtarr \mathbf{Sing} \bT X_*$$
and the natural map of simplicial spaces
$$ \ga\colon  \bT_{ver}^t \bS_* X_* \rtarr \bS \bT X_*.$$
The first is a weak equivalence. 
The second is a topological simplicial homotopy equivalence.    Therefore all spaces are both $\mathbf{Sing}$-connective and
$\bS$-connective.
\end{prop}
\begin{proof}  The last statement follows from the first two by unravelling \autoref{LAcon5} in our context. Its arrow $\ga$ is always a weak equivalence, hence so is its arrow $\OM \xi$ and therefore $\xi$ is always a weak equivalence.  

For the first map, apply $|-|$ and consider the following diagram of spaces:
$$\xymatrix{
 | \bT_{ver}^s\mathbf{Sing}_* X_* | \ar[r]^-{|\ga |} \ar[d]_{\epz}  & |\mathbf{Sing} \bT X_*|\ar[d]^{\epz} \\
  \bT X_*  \ar[r]_{\id} & \bT X_*.\\}
  $$
 The evaluation map $\epz$ on the right is a natural weak equivalence by classical arguments \cite[Theorem 16.6]{Mayss}. By \autoref{Tcommute}, $|\bT_{ver}^s\mathbf{Sing}_* X_* | \cong \bT | \mathbf{Sing}_* X_*|_*$ and $\epz$ here is this isomorphism post-composed by $\bT \epz_*$ and is therefore also a weak equivalence.  A check of definitions shows that the diagram commutes. It follows that $|\ga|$ is a weak equivalence.
 %We claim that there is an analogous evaluation map $\epz_2$ on the left which makes the diagram commute and is also a weak equivalence.  Granting the claim, it follows that 
 %$|\ga|$ is a weak equivalence. The (horizontal) simplicial set of $n$-simplicies of $\mathbf{Sing}_*X_*$ is $\mathbf{Sing}X_n$.   Its set of (vertical) $m$-simplices is the set $\mathbf{Map}(\Delta_m^t, X_n)$.   By \autoref{itworks}, we can restrict to the diagonal, so taking $m=n$, to obtain
%$| \mathbf{Sing}_* X_*|_2$.   But then evaluation of maps 
% $\mathbf{Map}(\Delta_n^t, X_n)\times \DE_n^t \rtarr X_n$ gives the map $\epz_2$, and it is a weak equivalence by an argument similar to the proof of the cited classical result.  \pmar{Check?} A check of definitions shows that the diagram commutes. For reassurance, note that all simplicial sets are Reedy cofibrant. 

For the second map, apply $\bT$ levelwise and consider the following diagram, where $\pi\colon \bS X \rtarr c_*X$ for a space $X$ is the homotopy equivalence defined in the proof of \autoref{new0}.
$$\xymatrix{
 \bT\bT_{ver}^t \bS_* X_* \ar[r]^-{\bT\ga} \ar[d]_{\bT \bT_{ver}^t \pi_*} & \bT \bS \bT X_* \ar[d]^{\bT\pi} \\
\bT \bT_{ver}^t c_*X_*  \ar[r]_{\iso} & \bT c_* \bT X_*.\\} 
  $$
At the bottom left, $c_*X_*$ denotes the levelwise constant bisimplicial space at the simplicial space $X_*$; at the bottom right,  $\bT c_* \bT X_* \iso c_*\bT X_*$ is the constant simplicial space at $\bT X_*$. By inspection, $\bT_{ver}^t  c_*\iso c_* \bT$, giving the bottom isomorphism.  The diagram commutes and the vertical arrows are equivalences, hence $\ga$ is an equivalence.
 \end{proof}

We conclude that the homotopical monadicity theorem as stated in  \autoref{HomMon2} and 
\autoref{Onew2} applies to give the following complement to \autoref{new?} for the adjunctions $(|-|,\mathbf{Sing})$ and $(\bT, \bS)$. 

\begin{thm}  The homotopy category of spaces is equivalent to both the homotopy category of  $\GA_s$-algebras in simplicial sets and the homotopy category of $\GA$-algebras in simplicial spaces, where $\GA_s = \mathbf{Sing}\com |-|$, $\GA = \bS\bT$, and the latter two homotopy categories are defined using linked weak equivalences as in \autoref{Onew}.
\end{thm} 

\subsection{Comments on discrete simplicial spaces} 

As already noted, the unit $\Id\rtarr \bU\bD$ is the identity: the underlying simplicial set of a simplicial set regarded as a discrete simplicial space is itself.  
Write $\epz^d$ for the counit of the adjunction $(\bD,\bU)$.  Levelwise, it is given by the identity continuous map $\epz^d\colon X^{\de}\rtarr X$ from the discretization of a space to itself.   For a space $X$, $\epz$ is  the composite
$$\xymatrix@1{  |\mathbf{Sing} X| = \bT\bD\bU\bS X \ar[r]^-{\bT \epz^d} & \bT \bS X \ar[r]^-{\epz} & X.\\} $$
Although $\epz^d$ is rarely a weak equivalence, \autoref{new?} implies that $\bT\epz^d \bS$ is a weak equivalence.  Since $\bS$ preserves weak equivalences,  the diagram
$$\xymatrix{
\bS\bT\bD\bU\bS X \ar[r]^-{\bS\bT\epz^d } & \bS\bT \bS X\\
\bD\bU \bS X \ar[r]_-{\epz^d}  \ar[u]^{\et} & \bS X \ar[u]_{\et} \\} $$
then shows that  the bottom map $\epz^d$ is a weak equivalence since the maps $\et$ are weak equivalences, again by \autoref{new?}.  We find this quite surprising!  Taking $X = \bT Y$ for a simplicial space $Y$, we have a diagram
$$\xymatrix{
\bD\bU  \bS\bT Y \ar[r]^-{\epz^d} & \bS \bT Y \\
\bD\bU Y \ar[r]_{\epz^d} \ar[u]^{\bD\bU  \et} & Y. \ar[u]_{\et}\\}$$
The top arrow and the right vertical arrow are weak equivalences, hence  $\bD\bU \et$ is a  weak equivalence if and only if the bottom arrow $\epz^d\colon \bD\bU  Y \rtarr Y$ is a weak equivalence.  But that is not an intuitive conclusion!  

For example, when $\et$ is a simplicial homotopy equivalence, it is reasonable to expect $\bD\bU \et$  to also be one.  Is it true that a weak equivalence between Kan simplicial spaces (defined the same way as Kan simplicial sets) is a simplicial homotopy equivalence?  This should be so with plausible model structures. 
But that may lead to $\epz^d\colon \bD\bU  Y \rtarr Y$ being a weak equivalence in unexpected and false generality.   

It is intriguing to ask if our framework sheds any light on classical results and questions, such as those of 
Milnor \cite{Milnor},  Friedlander and Mislin \cite{FP}, and others.  They study the homological behavior of $B\epz^d\colon B(G^{\de}) \rtarr  BG$ for certain Lie groups $G$, finding that there is a much closer relationship between the cohomologies than one might expect.

{
\section{Appendix.  Abstract monadic linking theory}\label{ablink}
\begin{defn}\label{oplax} An oplax map of monads $\bT\colon \bD \rtarr \bC$ is a functor $\bT$ together with a natural transformation 
$\ga\colon \bT\bD \rtarr \bC \bT$ such that the following diagrams commute:
$$ \xymatrix{ 
& \bT \ar[dl]_{\bT\et_{\bD}} \ar[dr]^{\et_{\bC}} &   & &  \bT\bD\bD \ar[r]^{\ga} \ar[d]_{\bT\mu_{\bD}} & \bC \bT \bD \ar[r]^{\bC\ga} & \bC\bC\bT \ar[d]^{\mu_{\bC}} \\
\bT\bD \ar[rr]_{\ga} & & \bC\bT                         & & \bT\bD \ar[rr]_{\ga} & & \bC\bT\\}$$
\end{defn}

The example of primary interest in this paper is given in \autoref{oplaxprop}.  There $\bC$ and $\bD$ are both monads associated to adjunctions, and the functor $\bT$ commutes with the left adjoint but not with the right adjoint.  Therefore verification that $\bT$ is an oplax map focuses on the right adjoint, as in \autoref{oplaxprop2}.   In that case, the right adjoints give functors to the respective categories of algebras, and then \autoref{link2} gives an example where the following definition applies.

\begin{defn}\label{CDlink} Let $\bT\colon \bD\rtarr \bC$ be an oplax map of monads.  A $\bC$-algebra $(X,\tha)$ is linked to a $\bD$-algebra $(Y,\ps)$ if there is a map $\de\colon \bT Y \rtarr X$ such that the following diagrams commute: 
$$\xymatrix{
\bT Y \ar[rr]^{\de} \ar[d]_{\bT\et_{\bD}} & & X \ar[d]^{\et_{\bC}}    & &  \bT\bD Y \ar[r]^{\ga} \ar[d]_{\bT\ps} & \bC \bT Y \ar[r]^{\bC\de} & \bC X \ar[d]^{\tha} \\
 \bT\bD Y \ar[r]_{\ga} & \bC\bT Y \ar[r]_{\bC\de} & \bC X      & & \bT Y \ar[rr]_{\de} & & X \\}$$
\end{defn}
\begin{rem} When $\bT$ is the identity $\bC\rtarr \bC$, with $\ga = \id$, $\de$ is just a map of $\bC$-algebras, so linking generalizes maps of $\bC$-algebras to the context of oplax maps of monads. In an intermediate case, we might have an oplax map $\bT\colon \bD\rtarr \bC$ such that  $\bT\bD=\bC \bT$.  For a $\bD$-algebra $(Y,\ps)$ we would then have the $\bC$-action $\bT\ps\colon \bC\bT Y = \bT\bD Y \rtarr \bT Y$ on $\bT Y$, but this fails in our examples, forcing us to the theory of linkings.  In another intermediate case, we might have $X = \bT Y$ with $\de = \id$.  That is the case given for example in \autoref{link} and \autoref{exmp1}.
\end{rem}

\begin{rem} The example of immediate interest is given by \autoref{link2}.  There we take $\bC$ to be the adjunction monad $\GA = \OM\SI$ on $\sT$ and take $\bD$ to be the analogous adjunction monad $\GA_* =\OM_*\SI_*$ on $s\sT$. The $*$ indicates levelwise application of a functor on $\sT$.  We take $\bT$ to be realization $s\sT  \rtarr \sT$ and $\ga\colon \bT\bD\rtarr \bC\bT$ to be $\ga\colon \bT\OM_*\SI_* \rtarr \OM \bT\SI_* \iso \OM \SI \bT$. 
We take $Y = \OM_* Z_*$ for $Z_*\in s\sS$ and take $X = \OM \bT Z_*$.  
We take $\de = \ga\colon \bT\OM_* \rtarr \OM\bT$, ignoring $\SI_*$, as earlier in the paper.  The diagrams of \autoref{oplax2} then give the diagrams of \autoref{CDlink}. In particular, taking $Z_* = B_*(\GA,\GA, Y) = \OM_*B_*(\SI,\GA,Y)$, we see that $B_*(\GA,\GA,Y)$ is linked to $\OM\bE Y$ for any $Y\in \GA[\sT]$.
\end{rem}

There is an elaboration of this monadic context in Srinivasan \cite{Ajay},  where it is applied in motivic homotopy theory.
}

\bibliographystyle{alpha}
\bibliography{references}

\end{document}